\documentclass[a4paper,12pt]{article}
\usepackage[margin=1in]{geometry}
\usepackage{appendix}
\usepackage{amsmath,amsthm,amsfonts,amssymb}
\usepackage{xcolor}
\usepackage{mathtools}

\mathcode`l="8000
\begingroup
  \lccode`~=`l
  \lowercase{\endgroup\def~{\ell}}
  
\usepackage{hyperref}

\newcommand{\mybox}[1]{%
  \begingroup
  \setlength{\fboxsep}{1pt}%
  \fbox{\small #1}%
  \endgroup
}

\newtheorem{theorem}{Theorem}
\newtheorem{lemma}{Lemma}

\newtheorem{remark}{Remark}
\newtheorem{question}{Question}

\newenvironment{grid}
  {\begingroup\footnotesize
   
   \setlength{\arraycolsep}{0pt}
   \left[\begin{array}{@{}*{10}{c}@{}}}
  {\end{array}\right]\endgroup}
\newcommand{\void}{\phantom{\times}}
\newcommand{\myC}{
    \begin{grid}
        \times & \void & \times \\
        \times & \square & \times \\
        \times & \times & \times
    \end{grid}
}

\newcommand{\myD}{
    \begin{grid}
        \times & \void & \void \\
        \times & \square & \times \\
	\times & \times & \times
    \end{grid}
}
\newcommand{\myE}{
    \begin{grid}
        \times & \square & \times \\
        \times & \times & \times
    \end{grid}
}

\newcommand{\myF}{
    \begin{grid}
        \void & \square & \void \\
        \times & \times & \times
    \end{grid}
}
\newcommand{\myG}{
    \begin{grid}
        \times & \square & \void \\
        \times & \times & \times
    \end{grid}
}
\newcommand{\myH}{
    \begin{grid}
        \times & \void & \void \\
        \times & \square & \void \\
        \times & \times & \times
    \end{grid}
}
\newcommand{\myL}{
    \begin{grid}
        \times & \square & \void & \void \\
        \times & \times & \times & \times
    \end{grid}
}
\newcommand{\myM}{
    \begin{grid}
        \times & \void & \void & \void \\
        \times & \square & \void & \void \\
        \times & \times & \times & \times
    \end{grid}
}
\newcommand{\myP}{
    \begin{grid}
        \square & \square & \void \\
        \times & \times & \times
    \end{grid}
}
\newcommand{\myQ}{
    \begin{grid}
        \times & \square & \square \\
        \times & \times & \times
    \end{grid}
}
\newcommand{\myR}{
    \begin{grid}
        \times & \void & \void & \void \\
        \times & \square & \square & \void \\
        \times & \times & \times & \times
    \end{grid}
}
\newcommand{\myS}{
    \begin{grid}
        \times & \void & \void \\
        \times & \square & \square \\
        \times & \times & \times
    \end{grid}
}
\newcommand{\myT}{
    \begin{grid}
        \times & \square & \square & \void \\
        \times & \times & \times & \times
    \end{grid}
}
\newcommand{\myU}{
    \begin{grid}
        \void & \square & \square & \void \\
        \times & \times & \times & \times
    \end{grid}
}
\newcommand{\myV}{
    \begin{grid}
        \times & \square & \square & \square \\
        \times & \times & \times & \times
    \end{grid}
}
\newcommand{\myW}{
    \begin{grid}
        \times & \void & \void & \void \\
        \times & \square & \square & \square \\
        \times & \times & \times & \times
    \end{grid}
}

\newcommand{\myX}{
    \begin{grid}
        \times & \square & \square & \times \\
        \times & \times & \times & \times
    \end{grid}
}

\newcommand{\myY}{
    \begin{grid}
        \times & \void & \void & \void \\
        \times & \square & \square & \times \\
        \times & \times & \times & \times
    \end{grid}
}

\newcommand{\myZ}{
    \begin{grid}
        \times & \void & \void & \times \\
        \times & \square & \square & \times \\
        \times & \times & \times & \times
    \end{grid}
}

\title{A convolutional approach to bounding the number of polyominoes}

\author{Vuong Bui\thanks{
Swinburne Vietnam, FPT University, Hanoi, 80 Duy Tan Street, Hanoi 100000, Vietnam (\texttt{bui.vuong@yandex.ru})}}

\date{}

\begin{document}
\allowdisplaybreaks

\maketitle

\begin{abstract}
Although known lower bounds for the growth rate $\lambda$ of polyominoes, or Klarner's constant, are already close to the empirically estimated value $4.06$, almost no conceptual progress on upper bounds has occurred since the seminal work of Klarner and Rivest (1973). Their approach, based on enumerating millions of local neighborhoods (also called ``twigs'') yielded $\lambda \le 4.649551$, later refined by Barequet and Shalah (2022) to $\lambda \le 4.5252$ using trillions of configurations. The inefficiency lies in representing each polyomino as an almost unrestricted sequence of neighborhoods once the large set of neighborhoods is fixed.

We introduce a recurrence-based approach that constrains how local neighborhoods concatenate. Using a small system of convolution-type recurrences, we obtain $\lambda \le 4.5238$. The proof is short, self-contained, and hand-checkable. Despite the marginal numerical improvement, the main contribution is methodological: replacing trillions of configurations with a concise one-page system of recurrences. 

In addition, we present a new technique for rigorously bounding the growth of recurrences to any precision, applicable to a broad range of settings with nonnegative coefficients. The resulting upper bound even comes with a nice feature: a small set of parameters serves as the certificate for the bound, that is, one does not need to check more than a few arithmetic calculations to trust the bound.
\end{abstract}

\section{Introduction}
A finite edge-connected set of cells on the square lattice is called a \emph{polyomino}. Polyominoes have been studied for a long time and lie at the intersection of theoretical computer science, mathematics and statistical physics. The central question from the enumerative aspect is: what is the exponential growth of the number $A(n)$ of polyominoes with $n$ cells? The growth constant $\lambda$, also known as ``Klarner's constant'', is formally defined by
\[
    \lambda = \lim_{n\to\infty} \sqrt[n]{A(n)}.
\]
Note that two polyominoes are counted only once if one is a translate of the other.

The almost folklore reason for the existence of the limit is actually first mentioned in \cite{klarner1967cell} by Klarner: $A(l+m)\ge A(l)A(m)$ for every $l,m\ge 1$, which can be proved by a simple concatenation of every pair of polyominoes to obtain a unique polyomino of $l+m$ cells. By Fekete's lemma, this so-called supermultiplicativity of the sequence $A(n)$ implies that the limit exists and is precisely the supremum of $\sqrt[n]{A(n)}$ over all $n$.
The lower bound $\sqrt[n]{A(n)}$ for each $n$ is however not exactly as efficient as the approaches using ``transfer matrices''.
These matrices allow us to estimate the number of polyominoes where the width is at most some $w$ and the length can be arbitrary, see \cite{jensen2001enumerations} for an example. Restricting polyominoes in this way, one can build the polyominoes column by column and we have a matrix of the transitions between the configuration of one column to the next column. However, restricting the problem to bounded-width columns does not suffice for an upper bound. While lower bounds involve linear recurrences, it would be interesting to see that in this article upper bounds could be obtained by studying convolution-type recurrences instead. The state of the art for lower bounds is transfer-matrix method for polyominoes on twisted cylinders \cite{barequet2016lambda}, where $\lambda$ is shown to be strictly larger than $4$ with $\lambda\ge 4.0025$ and quite close to the estimated (without proof) value $\lambda\approx 4.06$ (see \cite{jensen2003counting}).

Although the lower bound $4.0025$ is satisfactorily close to the estimated value of $\lambda$, the other side has not seen much progress. Note that we do not care a lot about the computability, but the actual practicality of the approach. For example, although computing $A(n)$ for large enough $n$ gives an arbitrarily good lower bound to $\lambda$, transfer matrix appears to be a better approach in terms of the amount of computation involved. (Another approach is that if one could prove the widely believed monotonicity of $A(n+1)/A(n)$, we can obtain an even better lower bound with $A(70)$.)
In fact, the computability of $\lambda$ can be deduced from the recent result in \cite{bui2024asymptotic}.
Specifically, \cite{bui2024asymptotic} showed that there exist explicit positive constants $c,t$ so that $A(n)\ge cn^{-t\log n} \lambda^n$ for every $n$. Together with the well-known corollary $A(n)\le\lambda^n$ of Fekete's lemma, we have
\[
    \sqrt[n]{A(n)} \le \lambda \le \sqrt[n]{\frac{n^{t\log n}}{c}A(n)}.
\]
The ratio of upper bounds and lower bounds converges to $1$ as $n\to\infty$.
However, the explicit values of $c,t$ in \cite{bui2024asymptotic} are only of theoretical interest, and to bound $\lambda$ to the desired precision, it requires $A(n)$ for large $n$. It is almost hopeless, as we know only up to $A(70)$, which is already a breakthrough on its own \cite{barequet2024counting} (a delicate variant of the transfer-matrix method is given there).

On the other hand, the asymptotic behavior of the growth constant of polycubes (the higher dimensional version of polyominoes) when the dimension goes to infinity is already known \cite{barequet2010formulae}. However, the technique cannot be brought to lower dimensions, where more delicate analysis is needed.

\paragraph{A brief history of upper bounds.}
Eden \cite{eden1961two} provided the first ever upper bound $\lambda\le 6.75$. One can remember it as the growth rate of the number of (rooted ordered) ternary trees. As each vertex has at most three children and at most one father, the degree of each vertex is at most four, which is also the coordination number of the square lattice. In fact, one can injectively map a polyomino to such a tree. However, a lot of ternary trees have no corresponding polyominoes, since several vertices may coincide when we map them to the square lattice. To reduce the number of such cases, one can take into account a broader context than the nearest neighborhood, e.g., more distant vertices.

Extending Eden's approach, Klarner and Rivest \cite{klarner1973procedure} forbid, for every cell, some certain positions in the neighborhood to have cells and we are allowed to extend through the remaining positions only. Such a configuration of forbidden and extendable positions forms a so-called ``twig''. By carefully designing a set of twigs, we can injectively map a polyomino of $n$ cells into a unique sequence of twigs. The order can be performed in a breadth-first-search manner.

Initially, Klarner and Rivest designed the set with $5$ twigs, which readily gives the bound $\lambda\le 5$. By assigning weights to each twig based on the actual number of forbidden and extendable cells in each twig (which are not uniform), one can obtain better bounds with $\lambda\le 4.83$. Previously, this bound is the only one that can be verified manually without computer assistance. The work \cite{bui2025bounding} reproved the bound in a cleaner way with recurrences.

Subsequently, Klarner and Rivest showed that we can add several twigs at the same time instead of one small twig by one small twig. Although it does not complicate the framework, it captures the neighborhood of each cell to a farther distance, at the cost of the number of twigs growing exponentially. In particular, by studying \emph{millions} of twigs, they obtained
\[
    \lambda\le 4.649551.
\]
The bound stood for almost half a century until the recent work by Barequet and Shalah \cite{barequet2022improved}. It employs the computing power of today and generates \emph{tens of trillions}\footnote{more precisely 23,521,568,438,976} of twigs to yield
\[
    \lambda\le 4.5252.
\]
A huge amount of computation on supercomputers was involved and it may be the case that different implementations may slightly disagree at some points. In particular, Barequet and Shalah pointed out that some of the computations done in Klarner and Rivest's work were not carried out in a correct way, which leads to a bit of mismatch between the original result and the reproduced result. Fortunately, the errors are marginal. Meanwhile, the computation in this article involves a handful of explicit rational numbers only.

\paragraph{Our convolutional approach.}
In this article, we propose a more efficient approach that can obtain the same bounds as Klarner--Rivest and Barequet--Shalah without using a huge amount of computation.
While we agree on using the information of neighborhoods, we do it in a different manner. In particular, for a given kind of neighborhood $\mathcal N$, we denote by $\mathcal N(n)$ the number of polyomino--position pairs so that the polyomino has $n$ cells and the marked position of the polyomino has the neighborhood of type $\mathcal N$. The perspective was used in \cite{bui2025bounding} but was not exactly successful.
In that article, $G(n)$ is the number of polyomino--cell pairs $(P,c)$ so that $c$ is a cell in a polyomino $P$ with $n$ cells and if $c$ is the white cell in 
\[
    \begin{grid}
        \times & \mybox{c} & \void \\
        \times & \times & \times
    \end{grid},
\]
then no cell of $P$ is allowed to be at the crossed positions. More precisely, no cell is allowed to be adjacent to $c$ to the left or from below, and no cell of the row below $c$ is in the adjacent columns.\footnote{This is actually the $L$-shape that was used to build Klarner--Rivest twigs.} On the other hand, we do not constrain any other positions, say, the positions to the right of $c$ and above $c$ may or may not have cells in $P$.

Roughly speaking, $G(n)$ counts the number of occurrences of the neighborhood over all polyominoes with $n$ cells.
For example, the following $5$-cell polyomino
\[
P = \begin{grid}
    \square & \square & \square \\
    \square & \void & \square
\end{grid}
\]
contributes $2$ to $G(5)$, corresponding to $(P,u)$ and $(P,v)$ where $u,v$ are two cells in the bottom, as illustrated below:
\[
\begin{grid}
    \void & \square & \square & \square & \void \\
    \times & \mybox{u} & \times & \mybox{v} & \void \\
    \times & \times & \times & \times & \times
\end{grid}.
\]
Later on we also use neighborhoods that involve more than one cell. For example, the following neighborhood has two cells:
\[
    \begin{grid}
        \void & \square & \void \\
        \times & \square & \times \\
        \times & \times & \times
    \end{grid}.
\]
Although we often do not want to name neighborhoods every time, we still want to count their occurrences. Therefore, for the above neighborhood we denote by $\begin{grid}
        \void & \square & \void \\
        \times & \square & \times \\
        \times & \times & \times
    \end{grid}_n$ the number of occurrences of the neighborhood over all polyominoes with $n$ cells. In other words, we can write $G(n)=\myG_n$. Note that we do not have to specify the marked cell $c$ explicitly in $\begin{grid}
        \void & \square & \void \\
        \times & \square & \times \\
        \times & \times & \times
    \end{grid}$, since $\begin{grid}
        \void & \mybox{c} & \void \\
        \times & \square & \times \\
        \times & \times & \times
    \end{grid}_n = \begin{grid}
        \void & \square & \void \\
        \times & \mybox{c} & \times \\
        \times & \times & \times
    \end{grid}_n$. The polyomino $P=\begin{grid}
    \square & \square & \square \\
    \square & \void & \square
\end{grid}$ also contributes $2$ to $\begin{grid}
        \void & \square & \void \\
        \times & \square & \times \\
        \times & \times & \times
    \end{grid}_5$, like the case of the neighborhood $\myG$ for $G(n)$.
    The two occurrences are $\begin{grid}
    \void & \blacksquare & \square & \square \\
    \times & \blacksquare & \times & \square \\
    \times & \times & \times & \void 
\end{grid}$ and $\begin{grid}
    \square & \square & \blacksquare & \void \\
    \square & \times & \blacksquare & \times \\
    \void & \times & \times & \times
\end{grid}$.
    In fact, one can observe that $\begin{grid}
        \void & \square & \void \\
        \times & \square & \times \\
        \times & \times & \times
    \end{grid}_n\le \myF_{n-1}$ since the conditions for the former are stricter than those for the latter.

It was shown in \cite{bui2025bounding} that $G(0)=G(1)=1$ and for $n\ge 2$,
\[
    G(n) \le 2\sum_{m=1}^{n-1} G(m)G(n-1-m).
\]
Note that an upper bound on the growth rate of $G(n)$ is also an upper bound on $\lambda$ since $G(n)\ge A(n)$, due to the fact that the left-most cell on the bottom-most row of a polyomino has the neighborhood of Type $G$. In fact, the growth rates of $A(n)$ and $G(n)$ are identical since
\[
    A(n)\le G(n)\le nA(n).
\]

We cannot go so far with only one kind of neighborhood,\footnote{In fact, there is another kind of neighborhood in \cite{bui2025bounding} but it largely plays as a supporting role for a more readable proof rather than an essential element.} since the resulting bound is $\lambda\le 4.83$, the same manual bound as in Klarner and Rivest's work. 
Roughly speaking, the approach in this article, after starting with a neighborhood, adds some cells to it, and then decompose the result into smaller neighborhoods, which may be of different types. By doing this, we actually let the neighborhood types interact with (and constrain) each other in a meaningful way (which is quite different from Klarner--Rivest approach).
Integrating more and more neighborhood types, we can improve the best known upper bound with elementary proofs. In the following section, we replace millions of twigs by $6$ types of neighborhoods and improve the bound of Klarner and Rivest. Later on Section \ref{sec:barequet-shalah} goes further with a dozen more types only and already improves the state-of-the-art bound with
\[
	\lambda \le 4.5238.
\]

Although one can jump to Section \ref{sec:barequet-shalah} for the bound $4.5238$, we suggest the readers to get familiarized with the approach first in the following section  (Section \ref{sec:klarner-rivest}), where everything is described in more detail.
In particular, we present an approach to bounding the growth rates of convolution-type recurrences from above with certificates involving only rational numbers (that is, easily verifiable). A lower bound is also obtained in a somewhat dual process, but without certificates, unfortunately. The technique can be applied elsewhere, even to much more complicated recurrences.
In addition, for the readers who are already familiar with Klarner--Rivest twigs, we provide an explanation why our approach could be more efficient at the end of Section \ref{sec:klarner-rivest}.

While one can see that the technique in this article can be readily applied to animals of other lattices, it may be too early to say something beyond that. However, we believe that the technique here can be applied to some other kinds of combinatorial problems with similar local constraints.

\section{Improving the bound of Klarner and Rivest} \label{sec:klarner-rivest}
We improve the Klarner and Rivest bound $\lambda\le 4.649551$ by showing that
\[
    \lambda\le 4.63.
\]
Instead of using millions of twigs, we use only a handful of neighborhoods types but let them interact as promised.
In particular, we define several kinds of neighborhoods similar to $G$ as follows:
\begin{center}
\begin{tabular}{ |c|c|c|c|c|c| } 
 \hline
 E & F & G & H & L & M \\ 
 \hline
 $\myE$ & $\myF$ & $\myG$ & $\myH$ & $\myL$ & $\myM$ \\ 
 \hline
\end{tabular}
\end{center}

We let $E(n),F(n),G(n),H(n),L(n),M(n)$ denote the number of polyomino-cell pairs $(P,c)$ so that $P$ has $n$ cells and if the cell $c\in P$ is at the white square then no cell of $P$ is allowed to be at the crossed positions of the corresponding neighborhood type.
The initial values are obvious:
\[
    E(1)=F(1)=G(1)=H(1)=L(1)=M(1)=1.
\]
In this article, we do not work with $E(n),F(n),G(n),H(n),L(n),M(n)$ for $n\le 0$. In fact, we can safely assume them to be zero without problems.\footnote{The readers of \cite{bui2025bounding} may be confused a bit as $G(0)$ was set to be $1$ in \cite{bui2025bounding} for some convenience. In particular, it serves as the default value for an empty polyomino, which is not considered in this article.} In other words, we do not write $\sum_{\substack{i,j\ge 1\\ i+j=n}}$ but simply write $\sum_{i+j=n}$ for short.
\begin{lemma} \label{lem:first}
For $n\ge 2$,
\begin{align*}
E(n) &\le F(n-1), \\
F(n) &\le G(n) + \sum_{i+j=n} G(i) H(j), \\
G(n) &\le F(n-1)+G(n-1)+\sum_{i+j=n-1} G(i) L(j), \\
H(n) &\le 2G(n-1)+ \sum_{i+j=n-1} E(i) L(j), \\
L(n) &\le F(n-1)+H(n-1)+ \sum_{i+j=n-1} G(i) M(j), \\
M(n) &\le G(n-1)+H(n-1)+ \sum_{i+j=n-1} E(i) M(j).
\end{align*}
\end{lemma}
\begin{proof}
The inequality $E(n)\le F(n-1)$ is fairly simple by
\[
    E(n) = 
    \begin{grid}
        \times & \mybox{c} & \times \\
        \times & \times & \times
    \end{grid}_n
    =
    \begin{grid}
        \void & \mybox{u} & \void \\
        \times & \mybox{c} & \times \\
        \times & \times & \times
    \end{grid}_n = \begin{grid}
        \void & \mybox{u} & \void \\
        \times & \times & \times \\
        \times & \times & \times
    \end{grid}_{n-1} \le \begin{grid}
        \void & \mybox{u} & \void \\
        \times & \times & \times \\
        \void & \void & \void
    \end{grid}_{n-1} = F(n-1).
\]
Indeed, the marked cell $c$ has only one possible neighbor, therefore, there must be a square $u$ on top of $c$ (since $n\ge 2$). Now the cell $c$ becomes isolated from the rest of the polyomino and we can exclude it from the polyomino with the new marked cell being $u$. The neighborhood of $u$ has $6$ forbidden positions, but for the sake of upper bounds, we discard the more distant positions at the bottom row and only forbid the $3$ positions on the row below $u$. This neighborhood is of Type $F$ with one cell less (due to removing $c$). Therefore, $E(n)\le F(n-1)$.

The situation for $F(n)$ is a bit more complicated as follows:
\begin{align*}
    F(n) &=     \begin{grid}
        \void & \mybox{c} & \void \\
        \times & \times & \times
    \end{grid}_n \\
    &= \begin{grid}
        \times & \mybox{c} & \void \\
        \times & \times & \times
    \end{grid}_n + \begin{grid}
        \mybox{u} & \mybox{c} & \void \\
        \times & \times & \times
    \end{grid}_n \\
    &\le \begin{grid}
        \times & \mybox{c} & \void \\
        \times & \times & \times
    \end{grid}_n + \sum_{i+j=n} \begin{grid}
        \void & \times &\void \\
        \mybox{u} & \times & \times \\
        \times & \times &\times
    \end{grid}_i \begin{grid}
        \void & \times & \void & \void \\
        \times & \times & \mybox{c} & \void \\
        \void & \times & \times & \times
    \end{grid}_j \\
    &\le \begin{grid}
        \times & \mybox{c} & \void \\
        \times & \times & \times
    \end{grid}_n + \sum_{i+j=n} \begin{grid}
        \void & \times & \void \\
        \mybox{u} & \times & \void \\
        \times & \times &\void
    \end{grid}_i \begin{grid}
        \void & \times & \void & \void \\
        \void & \times & \mybox{c} & \void \\
        \void & \times & \times & \times
    \end{grid}_j \\
    &= \begin{grid}
        \times & \square & \void \\
        \times & \times & \times
    \end{grid}_n + \sum_{i+j=n} \begin{grid}
        \times & \square & \void \\
        \times & \times & \times
    \end{grid}_i \begin{grid}
        \times & \void & \void \\
        \times & \square & \void \\
        \times & \times & \times
    \end{grid}_j \\
    &= G(n) + \sum_{i+j=n} G(i) H(j).
\end{align*}
The position to the left of the marked cell $c$ can either be empty or has a cell $u$, for which the state is $\begin{grid}
        \times & \mybox{c} & \void \\
        \times & \times & \times
    \end{grid}$ or $\begin{grid}
        \mybox{u} & \mybox{c} & \void \\
        \times & \times & \times
    \end{grid}$, respectively.
The first case corresponds to $G(n)$ trivially. For the second case, we name the adjacent positions $1,2,3,4$ as in
$\begin{grid}
        \void & \mybox{\scriptsize 2} & \mybox{\scriptsize 3} & \void \\
        \mybox{\scriptsize 1} & \mybox{u} & \mybox{c} & \mybox{\scriptsize 4} \\
        \void & \times & \times & \times
    \end{grid}_n$.
We further decompose $P$ into two smaller polyominoes, one contains $u$ and the other contains $c$. The positions $1,2$ if not empty are allocated to the polyomino of $u$. Likewise, the positions $3,4$ are allocated to the one for $c$. The neighborhood of $u,c$ will be
$\begin{grid}
        \void & \times &\void \\
        \mybox{u} & \times & \times \\
        \times & \times &\times
    \end{grid}$ and $\begin{grid}
        \void & \times & \void & \void \\
        \times & \times & \mybox{c} & \void \\
        \void & \times & \times & \times
    \end{grid}$, which will be reduced to respectively $\begin{grid}
        \void & \times & \void \\
        \mybox{u} & \times & \void \\
        \times & \times &\void
    \end{grid}_i$ and $\begin{grid}
        \void & \times & \void & \void \\
        \void & \times & \mybox{c} & \void \\
        \void & \times & \times & \times
    \end{grid}$ for an upper bound. We rotate the former to match with Type $G$. The numbers of cells of the two polyominoes $i,j$ sum up to $n$. The equalities and inequalities follow.

    We remark that the choice of the decomposition into two smaller polyominoes does not matter for the sake of the upper bound. In fact, we often have more than one way to decompose it as the polyomino can be well connected in the way that it is less like a tree. On the other hand, if we just put two polyominoes of $i,j$ cells ($i+j=n$) with the neighborhoods of types $G,H$ where the marked cells are put adjacent to each other, we may not obtain a polyomino with $i+j=n$ cells since the two polyominoes may share some cells. In either case, it does not matter for the direction of the inequalities.

The situation for $G(n)$ is slightly more complicated than $F(n)$ as we expand through both neighbors of $c$ as follows:
\begin{align*}
    G(n) &=     \begin{grid}
        \times & \mybox{c} & \void \\
        \times & \times & \times
    \end{grid}_n \\
    &= \begin{grid}
        \void & \mybox{u} & \void \\
        \times & \mybox{c} & \times \\
        \times & \times & \times
    \end{grid}_n
    + \begin{grid}
        \void & \times & \void \\
        \times & \mybox{c} & \mybox{v} \\
        \times & \times & \times
    \end{grid}_n
    + \begin{grid}
        \void & \mybox{u} & \void \\
        \times & \mybox{c} & \mybox{v} \\
        \times & \times & \times
    \end{grid}_n \\
    &\le \begin{grid}
        \void & \mybox{u} & \void \\
        \times & \times & \times
    \end{grid}_{n-1}
    + \begin{grid}
        \times & \void \\
        \times & \mybox{v} \\
        \times & \times
    \end{grid}_{n-1}
    + \sum_{i+j=n-1}\begin{grid}
        \void & \mybox{u} & \times & \void \\
        \times & \times & \times & \times \\
        \times & \times & \times & \void
    \end{grid}_{i}  \begin{grid}
        \void & \times & \void \\
        \times & \times & \void \\
        \times & \times & \mybox{v} \\
        \times & \times & \times
    \end{grid}_{j} \\
    &\le \begin{grid}
        \void & \square & \void \\
        \times & \times & \times
    \end{grid}_{n-1}
    + \myG_{n-1}
    + \sum_{i+j=n-1}\begin{grid}
        \void & \mybox{u} & \times \\
        \times & \times & \times    \end{grid}_{i}  \begin{grid}
        \times & \void \\
        \times & \void \\
        \times & \mybox{v} \\
        \times & \times
    \end{grid}_{j} \\
    &= \myF_{n-1}
    + \myG_{n-1}
    + \sum_{i+j=n-1}\myG_{i} \myL_{j} \\
    &= F(n-1) + G(n-1) + \sum_{i+j=n-1} G(i) L(j).
\end{align*}

The three possible cases are:
$\begin{grid}
        \void & \square & \void \\
        \times & \square & \times \\
        \times & \times & \times
    \end{grid}$, $\begin{grid}
        \void & \times & \void \\
        \times & \square & \square \\
        \times & \times & \times
    \end{grid}$ and $\begin{grid}
        \void & \square & \void \\
        \times & \square & \square \\
        \times & \times & \times
    \end{grid}$, where for each of the two neighbors of $c$ we either have a cell or let it be empty. (We cannot let both be empty since $n\ge 2$.)
    The first two cases are reduced as in the first inequality. For the third case, we denote the cells and positions in the neighborhood of $c$ by $\begin{grid}
        \void & \mybox{\scriptsize 2} & \void & \void \\
        \mybox{\scriptsize 1} & \mybox{u} & \mybox{\scriptsize 3} & \void \\
        \times & \mybox{c} & \mybox{v} & \mybox{\scriptsize 4} \\
        \times & \times & \times & \void
    \end{grid}$. After removing $c$ from $P$, we partition $P\setminus\{c\}$, which has $n-1$ cells, into 2 polyominoes containing $u,v$ with the positions $1,2$ (resp. $3,4$) being allocated to the polyomino of $u$ (resp. $v$). Other steps are carried out accordingly.

The remaining inequalities will be verified similarly and we sketch it by dropping 1-2 more trivial steps as below:
\begin{align*}
    H(n) &= \myH_n \\
    &= \begin{grid}
        \times & \times & \void \\
        \times & \square & \square \\
        \times & \times & \times
    \end{grid}_n + \begin{grid}
        \times & \square & \void \\
        \times & \square & \times \\
        \times & \times & \times
    \end{grid}_n + \begin{grid}
        \times & \square & \void \\
        \times & \square & \square \\
        \times & \times & \times
    \end{grid}_n\\
    &\le \begin{grid}
            \times & \void \\
            \times & \square \\
            \times & \times
        \end{grid}_{n-1} + \begin{grid}
            \times & \square & \void \\
            \times & \times & \times
        \end{grid}_{n-1} + \sum_{i+j=n-1} \begin{grid}
            \times & \square & \times & \void \\
            \times & \times & \times & \times \\
            \times & \times & \times & \void
        \end{grid}_i \begin{grid}
            \void & \times & \void \\
            \times & \times & \void \\
            \times & \times & \square \\
            \times & \times & \times
        \end{grid}_j \\
        &\le 2G(n-1) + \sum_{i+j=n-1} E(i)L(j),\\
    L(n) &= \myL_n \\
    &=\begin{grid}
        \void & \square & \void & \void \\
        \times & \square & \times & \void \\
        \times & \times & \times & \times
    \end{grid}_n
    +
    \begin{grid}
        \void & \times & \void & \void \\
        \times & \square & \square & \void \\
        \times & \times & \times & \times
    \end{grid}_n
    +
    \begin{grid}
        \void & \square & \void & \void \\
        \times & \square & \square & \void \\
        \times & \times & \times & \times
    \end{grid}_n \\
	&\le \myF_{n-1}
	+ \myH_{n-1} + \sum_{i+j=n-1}
    \begin{grid}
        \void & \square & \times & \void \\
        \times & \times & \times & \times \\
        \times & \times & \times & \times
    \end{grid}_i \begin{grid}
        \void & \times & \void & \void \\
        \times & \times & \void & \void \\
        \times & \times & \square & \void \\
        \times & \times & \times & \times
    \end{grid}_j \\
    &= F(n-1) + H(n-1) + \sum_{i+j=n-1} G(i) M(j),\\
    M(n) &= \myM_n \\
    &= \begin{grid}
        \times & \square & \void & \void \\
        \times & \square & \times & \void \\
        \times & \times & \times & \times
    \end{grid}_n + \begin{grid}
        \times & \times & \void & \void \\
        \times & \square & \square & \void \\
        \times & \times & \times & \times
    \end{grid}_n + \begin{grid}
        \times & \square & \void & \void \\
        \times & \square & \square & \void \\
        \times & \times & \times & \times
    \end{grid}_n \\
    &\le \begin{grid}
        \times & \square & \void\\
        \times & \times & \times
    \end{grid}_{n-1} + \begin{grid}
        \times & \void & \void \\
        \times & \square & \void \\
        \times & \times & \times
    \end{grid}_{n-1} + \sum_{i+j=n-1} \begin{grid}
        \times & \square & \times & \void \\
        \times & \times & \times & \times \\
        \times & \times & \times & \times
    \end{grid}_i \begin{grid}
        \void & \times & \void & \void \\
        \times & \times & \void & \void \\
        \times & \times & \square & \void \\
        \times & \times & \times & \times
    \end{grid}_j \\
    &\le G(n-1) + H(n-1) + \sum_{i+j=n-1} E(i) M(j).\qedhere
\end{align*}
\end{proof}

Let us consider the upper bounds $\hat{E}(n), \hat{F}(n), \hat{G}(n), \hat{H}(n), \hat{L}(n), \hat{M}(n)$ of the original sequences by initializing these sequences similarly with 
\[
    \hat{E}(1)=\hat{F}(1)=\hat{G}(1)=\hat{H}(1)=\hat{L}(1)=\hat{M}(1)=1
\]
and let them mutually be recurrences of each other by replacing inequalities by equalities. In particular, for $n\ge 2$, we have
$\hat{E}(n) = \hat{F}(n-1)$, and similarly for others. To make the recurrences more consistent in the way that they involve only smaller indices, we substitute  $\hat{G}(n)$ in the representation of $\hat{F}(n)$ by the representation of $\hat{G}(n)$.\footnote{In principle, we do not have to substitute. However, not substituting will bring some complications to the process later. In particular, if we do not apply the substitution, we have to be careful with the order of the inequalities in Lemma \ref{lem:converge} to prove. Also, in the algorithm in Section \ref{sec:alternative}, we also have to be careful with the order of updating the variables in each iteration.} Finally, the system of recurrences for $n\ge 2$ is
\begin{align*} 
\hat{E}(n) &= \hat{F}(n-1), \\ 
\hat{F}(n) &= \hat{F}(n-1)+\hat{G}(n-1)+\sum_{i+j=n-1} \hat{G}(i) \hat{L}(j) + \sum_{i+j=n} \hat{G}(i) \hat{H}(j), \\ 
\hat{G}(n) &= \hat{F}(n-1)+\hat{G}(n-1)+\sum_{i+j=n-1} \hat{G}(i) \hat{L}(j), \\ 
\hat{H}(n) &= 2\hat{G}(n-1)+ \sum_{i+j=n-1} \hat{E}(i) \hat{L}(j), \\ 
\hat{L}(n) &= \hat{F}(n-1)+\hat{H}(n-1)+ \sum_{i+j=n-1} \hat{G}(i) \hat{M}(j), \\ 
\hat{M}(n) &= \hat{G}(n-1)+\hat{H}(n-1)+ \sum_{i+j=n-1} \hat{E}(i) \hat{M}(j). 
\end{align*}

Note that the dependency between the sequences is connected, so we have the same growth rates for all the sequences.

Let us analyze the generating functions $\phi_e(x),\phi_f(x),\phi_g(x),\phi_h(x),\phi_l(x),\phi_m(x)$ of the corresponding new sequences, e.g., $\phi_e(x)=\sum_{n\ge 1} \hat{E}(n) x^n$. It follows that
\begin{align*}
\phi_e(x) &= x + x\,\phi_f(x), \\
\phi_f(x) &= x + x\,\phi_f(x) + x\,\phi_g(x) + x\,\phi_g(x)\,\phi_l(x) + \phi_g(x)\,\phi_h(x), \\
\phi_g(x) &= x + x\,\phi_f(x) + x\,\phi_g(x) + x\,\phi_g(x)\,\phi_l(x), \\
\phi_h(x) &= x + 2x\,\phi_g(x) + x\,\phi_e(x)\,\phi_l(x), \\
\phi_l(x) &= x + x\,\phi_f(x) + x\,\phi_h(x) + x\,\phi_g(x)\,\phi_m(x), \\
\phi_m(x) &= x + x\,\phi_g(x) + x\,\phi_h(x) + x\,\phi_e(x)\,\phi_m(x).
\end{align*}
Usually, one would estimate the growth rates with the traditional singularity analysis, which may not be elementary enough for everyone. Moreover, rigorously bounding the growth rates of such a complex system may be not easy with singularity analysis. For the sake of upper bounds only, we use the following approach. To the best of our knowledge, the approach, despite being elementary, is new.
\begin{lemma}\label{lem:alternative}
If there are positive values $e,f,g,h,l,m$ and $x$ so that
\begin{align*}
e &\ge x + x f, \\
f &\ge x + x f + x g + x g l + g h, \\
g &\ge x + x f + x g + x g l, \\
h &\ge x + 2x g + x e l, \\
l &\ge x + x f + x h + x g m, \\
m &\ge x + x g + x h + x e m,
\end{align*}
then the growth rates of $E(n),F(n),G(n),H(n),L(n),M(n)$ are at most 
\[
    \frac{1}{x}.
\]
\end{lemma}
There is nothing special about our generating functions and the approach can be applied elsewhere.
The simple proof will be given in Section \ref{sec:alternative}.
\begin{theorem} \label{thm:4.63}
    \[
        \lambda\le 4.63.
    \]
\end{theorem}
\begin{proof}
For
\[
    x = \frac{1}{4.63} = \frac{100}{463},
\]
the following rational values
\[
(e,f,g,h,l,m) = 
\left(
\frac{34}{67},
\frac{139}{103},
\frac{67}{82},
\frac{101}{155},
\frac{95}{126},
\frac{106}{177}
\right)
\]
rigorously satisfy the inequalities in Lemma \ref{lem:alternative}. Therefore, the growth rates of all the sequences are bounded by $4.63$. The conclusion follows from the fact that the number of polyominoes with $n$ cells is bounded from above by $G(n)$.
\end{proof}

One can slightly push down the value $4.63$ using one tool or another. However, we would prefer to keep the rational values simple and $4.63$ is fairly close to the true rate. At any rate, the bound already improves the bound $4.649551$ by a significant margin.

\paragraph{Some comparisons with the approach using twigs.}
We relate the technique to the readers who are more familiar with Klarner--Rivest twigs.
There are some issues with twigs that can be improved. The twigs in the set do not interact with each other once the set is established (although they do before that). In particular, we do not constrain at all if some two twigs are allowed to be put adjacent or not. Even worse, the sequence of twigs is listed in the order of breadth-first search, which usually makes the $i$-th twig and the $(i+1)$-st twig have nothing to do with each other, e.g., they are too far apart. In other words, they basically count all the lists composed of elements of the set of twigs without any kind of constraints (other than the trivial constraint that the total numbers of forbidden cells and extendable cells agree with the actual number of cells).
It is true that adding several twigs at once would capture the dependency between twigs to some extent, but letting them interact with each other would hopefully let us express the dependency even further forward.

Another issue is that they construct larger twigs from smaller twigs and the smallest twigs involve the forbidden positions forming an $L$-shape. While the approach gives a consistent way to systematically construct twigs, there is no guarantee that other shapes than $L$-shapes cannot outperform significantly.
It is intuitive that the nearer neighborhood would decide the growth more than the more distant positions. Therefore, one should pay more attention to the former, which may be better off being other than the $L$-shape. A better choice earlier on is likely to save a lot of computation later. One can imagine it like lower-order terms of a Taylor series dominating the rest.

\paragraph{Our drawbacks.} Now come the drawbacks of our approach. The first point is that we still need to systematize it. Although it can be seen that we choose the neighborhoods by increasingly forbidding positions, starting from those closer to the marked cell first, it is still unclear how we can generate more and more useful neighborhoods. Usually these neighborhood types are naturally expanded and we stop expanding when we are satisfied with the resulting bound. However, dealing with bigger and bigger neighborhoods may be not so pleasant. The example given in Section \ref{sec:barequet-shalah} already requires a bit of effort to check all the manipulations. It is also doubtful if all possible neighborhoods could be useful, even if neighborhoods of small sizes seem to suggest that. We also need to invest the computational cost when the size of neighborhoods increases. Of course we should compare the cost with the convergence rate of the resulting upper bound. However, we still do not have an idea if it even converges to Klarner's constant.

\begin{question}
    If we are allowed to choose sets of neighborhoods of increasing sizes and we can derive the recurrences in an optimal way, will the resulting upper bound converge to Klarner's constant?
\end{question}

\section{An alternative method to singularity analysis} \label{sec:alternative}

We propose an approach to approximating the growth rate of the recurrences without using singularity analysis. We avoid singularity analysis since it is not quite suited for complex recurrences. Let us say, even in the simple case of a sequence being a recurrence of itself, how to rigorously locate the dominant singularity of the associated generating function is not straightforward. Meanwhile, our alternative approach provides a simple certificate and an algorithm to yield the certificate. In particular, we will prove Lemma \ref{lem:alternative} and give an algorithm to systematically compute solutions for Lemma \ref{lem:alternative}.
However, one can notice that our approach works for similar situations, that is, convolutional recurrences with nonnegative coefficients. An example is the more complicated system in Section \ref{sec:barequet-shalah}.

The approach is quite natural and straightforward. Given some $\mu_0$ and $\epsilon$, we would conclude for the growth rate $\mu$ of a sequence and the corresponding generating function $\phi$ that either $\phi(\frac{1}{\mu_0})$ diverges or $\phi(\frac{1}{\mu_0+\epsilon})$ converges. In other words, we decide whether $\mu\ge \mu_0$ or $\mu \le \mu_0+\epsilon$. Although strictly speaking it is not quite nice as the two intervals slightly overlap, it still suffices for an approximation algorithm using binary search. Note that our approach will use rational numbers only while the growth rate $\mu$ can have a complicated algebraic nature. Moreover, the behaviour of the function at singularity may be not very convenient to deal with, so we avoid it. That is to say using the method we cannot prove $\mu\le \mu_0$ for $\mu_0=\mu$ but we can prove $\mu\le \mu_0+\epsilon$ for any $\epsilon>0$. 

Apart from the fact that our approach can work with any convolutional recurrences with nonnegative coefficients, we would remark that in bounding the growth rate from below, we need at least one of the relations involving a convolution. However, it is our point to treat such a system, as otherwise, a pure linear system is already well studied in literature.

On the other hand, the process of bounding the growth rate from above yields a certificate for the bound, that is we do not need to run the whole program to trust the upper bound but just need to check a few conditions, provided that we have both the bound and the certificate. We will later ask the question whether it is possible to have a certificate for the lower bound.

To prove that the generating function $\phi(x)$ for some $x$ would be bounded or diverge, we maintain a sequence that converges to $\phi(x)$ from below, and always greater than some partial sum. In particular, we define the sequences
\[
\{e_n\}_{n\ge 1}, \{f_n\}_{n\ge 1},\{g_n\}_{n\ge 1},\{h_n\}_{n\ge 1},\{l_n\}_{n\ge 1},\{m_n\}_{n\ge 1}
\]
so that
\[
	e_1=f_1=g_1=h_1=l_1=m_1=x
\]
and for later indices we have
\begin{align*}
	e_{n+1} &= x + x f_n, \\
	f_{n+1} &= x + x f_n + x g_n + x g_n l_n + g_n h_n, \\
	g_{n+1} &= x + x f_n + x g_n + x g_n l_n, \\
	h_{n+1} &= x + 2x g_n + x e_n l_n, \\
	l_{n+1} &= x + x f_n + x h_n + x g_n m_n, \\
	m_{n+1} &= x + x g_n + x h_n + x e_n m_n.
\end{align*}

One can observe that the sequences are increasing by a simple induction. The base case is $e_2\ge x=e_1,\dots, m_2\ge x=m_1$. It remains to prove $e_{n+1}\ge e_n, \dots, m_{n+1}\ge m_n$ provided that they hold for smaller $n$. Indeed,
\begin{align*}
    e_{n+1}&=x+xf_n \ge x+xf_{n-1} = e_n,\\
    \dots&\dots,\\
    m_{n+1} &= x + x g_n + x h_n + x e_n m_n \ge x + x g_{n-1} + x h_{n-1} + x e_{n-1} m_{n-1} = m_n.
\end{align*}

It turns out that these sequences converge from below to the values of the corresponding generating functions at $x$ (if they are finite), as in the following lemma.
\begin{lemma} \label{lem:converge}
	For each sequence $s_n$ among the given sequences and for every $n$, we have
	\[
		\sum_{i=1}^n \hat{S}(i)x^i \le s_n \le \sum_{i=1}^\infty \hat{S}(i)x^i.
	\]
\end{lemma}

We will prove Lemma \ref{lem:converge} later in Appendix \ref{sec:converge}, whose proof is natural and nothing special, except for the statement of the lemma itself. For now, given Lemma \ref{lem:converge} we are ready to prove Lemma \ref{lem:alternative}.
\begin{proof}[Proof of Lemma \ref{lem:alternative}]
	By the condition of the statement, one can observe that $e,f,g,h,l,m \ge x$. In other words, $e\ge e_1, f\ge f_1, g\ge g_1, h\ge h_1, l\ge l_1, m\ge m_1$.
    Due to the induction of the sequences $e_n,f_n,g_n,h_n,l_n,m_n$, we have $e_n\le e, f_n\le f, g_n\le g, h_n\le h, l_n\le l, m_n\le m$ for every $n$. Meanwhile, the sequences $e_n,f_n,g_n,h_n,l_n,m_n$ converge to $\phi_e(x),\phi_f(x),\phi_g(x),\phi_h(x),\phi_l(x),\phi_m(x)$, respectively. Therefore, $\phi_e(x)\le e,\phi_f(x)\le f,\phi_g(x)\le g,\phi_h(x) \le h,\phi_l(x) \le l,\phi_m(x) \le m$. In other words, they are bounded, hence the growth rates of the corresponding sequences are at most $1/x$. The conclusion follows since $\hat{F}(n)\ge F(n), \dots, \hat{M}(n)\ge M(n)$.
\end{proof}

\subsection*{An algorithm to find the certificate}
Now comes the algorithm to find such a solution for Lemma \ref{lem:alternative}.
Given some $\mu_0$ and $\epsilon$, we would decide if the growth rate $\mu$ of the sequences is either greater than $\mu_0$ or at most $\mu_0+\epsilon$. In fact, we will run two algorithms concurrently and see which algorithm will stop first.

\begin{itemize}
    \item Checking if $\mu\le \mu_0+\epsilon$.

    Suppose $\mu\le \mu_0$. It follows that setting
    \[
        x=\frac{1}{\mu_0+\frac{\epsilon}{2}}
    \]
    would make the sequences $e_n,\dots,m_n$ stay bounded. In particular, the sequence $\Delta_e(n) = \phi_e(x) - e_n$ is nonnegative and decreasing to $0$. The same applies to the corresponding sequences $\Delta_f(n),\dots,\Delta_m(n)$. 
    Rewriting $\phi_e(x) = x + x\phi_f(x)$ using the newly defined terms, we have
    \[
        e_n + \Delta_e(n) = x + x(f_n + \Delta_f(n))
    \]
    
    When $n$ is large enough, the values $\Delta_e(n)$ and $\Delta_f(n)$ become small enough, which makes
    \[
        e_n \ge x' + x'f_n
    \]
    for $x'=\frac{1}{\mu_0+\epsilon}$, which is smaller than $x=\frac{1}{\mu_0+\frac{\epsilon}{2}}$ by a certain amount.
    When $n$ is large enough, the value $x'$ also satisfies the other inequalities. In other words, we obtain the solution $x'$ and $e_n,\dots,m_n$ of Lemma \ref{lem:alternative}.
    
    The algorithm is as simple as: Running for larger and larger $n$ until $x'$ satisfies the inequalities.
    
    If the algorithm stops, we conclude $\mu\le \mu_0+\epsilon$.
    
    \item Checking if $\mu\ge \mu_0$.

    Suppose $\mu>\mu_0$.\footnote{This is the negation of the previous assumption that $\mu\le\mu_0$. We can prove by computation that the generating functions diverge at $\frac{1}{\mu_0}$. However, it implies $\mu\ge \mu_0$ only.} It follows that setting
    \[
        x=\frac{1}{\mu_0}
    \]
    would make the generating functions diverge at $x$. In other words, the sequences $e_n,\dots,m_n$ surpass any threshold.  Particularly, $e_n,\dots,m_n$ are all larger than $\frac{1}{x}$ for some $n$ large enough. It follows that
    \[
      m_{n+1} = x + x g_n + x h_n + x e_n m_n \ge x + xe_nm_n\ge x + x\frac{1}{x}m_n = x + m_n.
    \]
    This makes the sequences diverge, as two consecutive elements for $n$ large enough differ by at least a fixed amount $x$.

    The algorithm is as simple as: Running for larger and larger $n$ until the sequences all surpass $\frac{1}{x}=\mu_0$.
    
    If the algorithm stops, we conclude $\mu\ge \mu_0$.
\end{itemize}

\paragraph{The algorithm.}The main algorithm is a combination of both algorithms:
\begin{quote}
We generate sequences $e_n,\dots,m_n$ with both starting points $x=\frac{1}{\mu_0+\frac{\epsilon}{2}}$ and $x=\frac{1}{\mu_0}$. In each step $n$,
\begin{itemize}
    \item if in the first setting we have $x'=\frac{1}{\mu_0+\epsilon}$ satisfies the inequalities in Lemma \ref{lem:alternative} with $e=e_n,\dots,m=m_n$ then we conclude $\mu\le \mu_0+\epsilon$ with the corresponding certificates $x',e_n,\dots,m_n$.
    \item if in the second setting all $e_n,\dots,m_n$ are at least $\frac{1}{x} = \mu_0$ then we conclude $\mu\ge \mu_0$.
\end{itemize}
The algorithm will eventually terminate for large enough $n$, since we always have either the assumption $\mu\le \mu_0$ or the assumption $\mu>\mu_0$.
\end{quote}

\begin{remark}
    When proving $m_{n+1}\ge x+m_n$, we have elements of the sequence $\{m_n\}_{n\ge 1}$ appearing on both sides of the recurrence of $m_{n+1}$, but it is not exactly necessary. In a system where the dependency between sequences is connected, we are still able to show such a thing by a circular dependence. For example, we can write
    \[
   	    l_{n+2} = x + x f_{n+1} + x h_{n+1} + x g_{n+1} m_{n+1} \ge x + xg_{n+1}m_{n+1}\ge x + g_{n+1},
    \]
    while
    \[
    	g_{n+1} = x + x f_n + x g_n + x g_n l_n \ge x+xg_nl_n\ge x + l_n.
    \]
    That is
    \[
        l_{n+2}\ge 2x+l_n.
    \]
\end{remark}

\subsection*{On implementations and certificates}
We have not discussed the convergence rate although in practice the number of iterations is usually not large, relative to the smallness of $\epsilon$. In fact, our main concern is the implementations of the numbers and the arithmetic operations, which may or may not make each iteration computationally expensive. In particular, we implicitly assume that operations are done in constant time, while in fact the most straightforward implementation using fractions would make the denominators explode very fast. This is due to the multiplications making the denominator kind of square in each iteration. One may consider floating point numbers instead. However, we cannot conveniently guarantee the convergence of the sequences and the termination of the algorithm. A lower bound produced by the algorithm could be not reliable then. However, in the case of an upper bound, we do not have to trust the algorithm, as long as the certificate it yields satisfies the inequalities in Lemma \ref{lem:alternative}. In principle, some inequalities could be violated by a small margin. But it should not be a big problem because we can slightly increase the upper bound to make the certificate valid.
In particular, if we want to obtain a certificate for $4.63$, we may set $\mu_0=4.628$ and $\epsilon=0.001$. We need to make $\mu_0+\epsilon$ smaller than $4.63$ by some margin, since floating point operations may make some inequalities slightly invalid. We obtain the certificate (using Python floating point numbers):
\[
(e,f,g,h,l,m)=
\begin{pmatrix}
0.5093160670307050 \\
1.3576230676346843 \\
0.8208902685491735 \\
0.6540068185686590 \\
0.7571242396493729 \\
0.6007616699697459
\end{pmatrix}.
\]

Since $\mu_0+\epsilon=4.629$ is smaller than $4.63$ by a certain margin, it allows us to mutate the above values a bit. In particular, one can choose the values in the proof of Theorem \ref{thm:4.63}, where the rational representations have much smaller denominators. On the other hand, if we set $\mu_0=4.629$, that is $\mu_0+\epsilon=4.63$, then the certificate obtained from the algorithm using floating points is no longer valid.

The approach of certificates allows us to use any mathematical programming tools to solve the program in Lemma \ref{lem:alternative}, just that it may be not so mathematically guaranteed as in our proposed algorithms. Meanwhile, we do not yet have a certificate-based approach to the lower bounds. One still has to rerun the whole algorithm to really trust a lower bound. Fortunately, we do not need a lower bound in this article.

Another approach to lower bounds is given in \cite{bui2024mixed}, which works for recurrences using not only summations but also the maximum operators (which is exactly a weak point of the methods based on generating functions). Adapting it to our situation, one would start with
\[
    \hat{G}(n)\ge \hat{G}(i)\hat{L}(j)
\]
for any $i,j$ so that $i+j=n-1$. We can safely set $0$ to the nonpositive indices of the sequences. In this way, $i,j$ can take nonpositive values as well. 

We also have
\[
    \hat{L}(j)\ge \hat{F}(j-1)\ge \hat{G}(j-1).
\]
In total,
\[
    \hat{G}(n)\ge \hat{G}(i)\hat{G}(j-1).
\]
Rewriting it yields
\[
    \hat{G}(i+j-2)\ge \hat{G}(i-2)\hat{G}(j-2)
\]
for any $i,j$.
In other words, the sequence $s_n=\hat{G}(n-2)$ is supermultiplicative. The growth rate of $s_n$ and $\hat{G}(n)$ are identical, and equal to the supremum of $\sqrt[n]{\hat{G}(n-2)}$. That is $\sqrt[n]{\hat{G}(n-2)}$ for any $n$ would be a lower bound for the growth rate. However, we cannot quite say that this is a certificate as we need to run an algorithm to compute $\hat{G}(n)$ for large $n$. An advantage of this approach is that $\hat{G}(n)$ grows exponentially with $n$ only, while the denominators of $g_n$ can grow doubly exponentially in principle.

Although a lower bound is not of interest for the recurrences in this article, we still pose the following question.
\begin{question}
    Can we have a simple and quickly verifiable certificate for a lower bound on the growth rate of convolution-type recurrences?
\end{question}

\section{Improving the bound of Barequet and Shalah}
\label{sec:barequet-shalah}
One could also wish for improving the state-of-the-art bound $\lambda\le 4.5252$, which was proved using tens of trillions of twigs. It would be interesting and would definitely make our approach a good proof-of-concept if we can still handle all recurrences manually. Note that when previous works reduced the bound from $4.649551$ to $4.5252$, the number of twigs increased \emph{from millions to tens of trillions} and we would say that we have reached the limit of possible improvement given the current computing power. Therefore, we shall also pay attention to how many more neighborhood types are to be considered in the new approach. In fact, there will be only a handful more, and the recurrences fit nicely in one page.

In this section, we prove
\[
    \lambda\le 4.5238
\]
using some enhancements over the previous section.
Before we even attempt to add more neighborhood types, we observe that we previously did not quite let the neighborhood types interact with each other a lot. In particular, we only tried to add cells that are adjacent to the marked cell. If we add more cells at a bit more distant positions, we can capture the nature of the square lattice better. 

Besides having more ``depth'' with the recurrences, one can of course add more neighborhood types to introduce more ``breadth'' to the approach.
One can forbid cells at more distant positions, but we do not have to do that at a very far distance to improve the bound $4.5252$. In particular, we mostly keep the positions to forbid, but consider more squares altogether. By allowing more cells to be included in the neighborhood state, we avoid partitioning the polyomino into too many smaller polyominoes, and it allows us to look slightly further to lower depths.

The neighborhood types we use are:
\begin{center}
\begin{tabular}{ |c|c|c|c|c|c|c|c|c| } 
 \hline
	C & D & E & F & G & H & P & Q & R \\ 
 \hline
	$\myC$ & $\myD$ & $\myE$ & $\myF$ & $\myG$ & $\myH$ & $\myP$ & $\myQ$ & $\myR$ \ \\ 
 \hline
\end{tabular}
\end{center}

\begin{center}
\begin{tabular}{ |c|c|c|c|c|c|c|c| } 
 \hline
	S & T & U & V & W & X & Y & Z \\
 \hline
	$\myS$ & $\myT$ & $\myU$ & $\myV$ & $\myW$ & $\myX$ & $\myY$ & $\myZ$ \\ 
 \hline
\end{tabular}
\end{center}

Obviously,
\[
	C(1)=D(1)=E(1)=F(1)=G(1)=H(1)=1
\]
and
\[
	P(1)=Q(1)=R(1)=S(1)=T(1)=U(1)=V(1)=W(1)=X(1)=Y(1)=Z(1)=0.
\]
We again use the convention that the sequences at nonpositive indices are all zeros.

\begin{lemma} \label{lem:full}
{\footnotesize
	For $n\ge 2$,
	\begin{align*}
	C(n) &\le E(n-1), \\
	D(n) &\le G(n-1), \\
	E(n) &\le F(n-1), \\
	F(n) &\le G(n) + P(n), \\
	G(n) &\le E(n) + Q(n), \\
	H(n) &\le D(n) + S(n), \\
	U(n) &\le \sum_{i+j=n} D(i)H(j)
	      + \sum_{i+j=n} S(i)D(j)
	      + \sum_{i+j=n} Y(i)R(j)
        + \sum_{i+j=n} W(i)Y(j)
	      + \sum_{i+j+k=n} U(i)Z(j)Z(k), \\
	T(n) &\le X(n) + V(n), \\
	P(n) &\le \sum_{i+j=n} E(i)H(j)
	      + \sum_{i+j=n} Q(i)D(j)
	      + \sum_{i+j=n} X(i)R(j)
	      + \sum_{i+j=n} V(i)Y(j)
	      + \sum_{i+j+k=n} U(i)Y(j)Z(k), \\
	Q(n) &\le G(n-1)
	      + \sum_{i+j=n-1} G(i)E(j)
	      + U(n-2)
	      + \sum_{i+j=n-2} T(i)G(j)
	      + \sum_{i+j=n-2} R(i)U(j), \\
	R(n) &\le Y(n) + W(n), \\
	S(n) &\le G(n-1)
	      + \sum_{i+j=n-1} E(i)E(j)
	      + T(n-2)
	      + \sum_{i+j=n-2} X(i)G(j)
	      + \sum_{i+j=n-2} Y(i)U(j), \\
	V(n) &\le S(n-1)
	      + \sum_{i+j=n-2} G(i)G(j)
	      + \sum_{i+j=n-2} T(i)E(j)
	      + \sum_{i+j=n-2} R(i)T(j), \\
	W(n) &\le S(n-1)
	      + \sum_{i+j=n-2} E(i)G(j)
	      + \sum_{i+j=n-2} X(i)E(j)
	      + \sum_{i+j=n-2} Y(i)T(j), \\
	X(n) &\le D(n-1) + G(n-2) + U(n-2), \\
	Y(n) &\le C(n-1) + G(n-2) + T(n-2), \\
	Z(n) &\le C(n-1) + E(n-2) + X(n-2).
	\end{align*}
}
\end{lemma}
The verification of Lemma \ref{lem:full} is given in Appendix \ref{sec:verification}.
\begin{theorem}
	\[
	    \lambda\le 4.5238.
	\]
\end{theorem}
\begin{proof}
	We let $\hat{C}(n),\hat{D}(n),\dots,\hat{Z}(n)$ be the upper bounds of $C(n),D(n),\dots,Z(n)$ in the same manner as in Section \ref{sec:klarner-rivest}. In particular, we initialize them the same values $\hat{C}(1)=C(1),\dots,\hat{Z}(1)=Z(1)$ and let them mutually be recurrences of each other by replacing inequalities by equalities. In other words, for $n\ge 2$, we have $\hat{C}(n) = \hat{E}(n-1)$ and similarly for others, for example,
	\[
		\hat{W}(n) = \hat{S}(n-1) + \sum_{i+j=n-2} \hat{E}(i)\hat{G}(j) + \sum_{i+j=n-2} \hat{X}(i)\hat{E}(j) + \sum_{i+j=n-2} \hat{Y}(i)\hat{T}(j).
	\]
    One can manually check by hand that there is no circular dependency, therefore, the sequences are well-defined.\footnote{We do not apply substitutions as in the recurrences of Section \ref{sec:klarner-rivest} since writing the full expansions would be too long. However, everything we have to do is to be careful with the order when proving and calculating. For example: we can first compute $\hat{C}, \hat{D}, \hat{E}, \hat{Q}, \hat{S}, \hat{V}, \hat{W}, \hat{X}, \hat{Y}, \hat{Z}, \hat{P}, \hat{U}$ from smaller indices, and then $\hat{G}, \hat{F}, \hat{H}, \hat{T}, \hat{R}$.} 

	We use capital letters $C,D,\dots,Z$ to denote the generating functions of these new sequences. For example, $C = \sum_{n\ge 1} \hat{C}(n)\zeta^n$. (We do not write $C(\zeta)$ in the place of $C$ as it is a bit too lengthy for the following equations and we use $\zeta$ for the variable as it is more distinguishable from the newly introduced $X$.) The generating functions satisfy
    {\small
	\begin{gather*}
		C = \zeta + \zeta E,\quad D = \zeta + \zeta G,\quad E = \zeta + \zeta F,\quad F = G + P,\quad G = E + Q,\quad H = D + S,\\
		P = E H + Q D + X R + V Y + U Y Z,\quad Q = \zeta G + \zeta(G E) + \zeta^{2}(U + T G + R U),\quad R = Y + W,\\
		S = \zeta G + \zeta E^{2} + \zeta^{2} T + \zeta^{2} X G + \zeta^{2} Y U,\quad T = X + V,\quad U = D H + S D + Y R + W Y + U Z^{2},\\
		V = \zeta S + \zeta^{2}(G^{2} + T E + R T),\quad W = \zeta S + \zeta^{2}(E G + X E + Y T),\\
		X = \zeta D + \zeta^{2}(G + U),\quad Y = \zeta C + \zeta^{2}(G + T),\quad Z = \zeta C + \zeta^{2}(E + X).
	\end{gather*}
    }
	Let $\zeta=1/4.5238=\frac{10000}{45238}$.
	Since the values
	\begin{gather*}
		c = \tfrac{871}{2500},\quad
		d = \tfrac{2157}{5000},\quad
		e = \tfrac{2879}{5000},\quad
		f = \tfrac{1003}{625},\quad
		g = \tfrac{4757}{5000},\quad
		h = \tfrac{1851}{2500},\\
		p = \tfrac{3267}{5000},\quad
		q = \tfrac{939}{2500},\quad
		r = \tfrac{599}{2500},\\
		s = \tfrac{309}{1000},\quad
		t = \tfrac{727}{2500},\quad
		u = \tfrac{633}{1250},\\
		v = \tfrac{621}{5000},\quad
		w = \tfrac{509}{5000},\\
		x = \tfrac{833}{5000},\quad
		y = \tfrac{689}{5000},\quad
		z = \tfrac{567}{5000}
	\end{gather*}
	satisfy
	\begin{gather*}
		c \ge \zeta + \zeta e,\quad 
		d \ge \zeta + \zeta g,\quad 
		e \ge \zeta + \zeta f,\quad 
		f \ge g + p,\quad 
		g \ge e + q,\quad 
		h \ge d + s,\\
		p \ge e h + q d + x r + v y + u y z,\quad 
		q \ge \zeta g + \zeta(g e) + \zeta^{2}(u + t g + r u),\quad 
		r \ge y + w,\\
		s \ge \zeta g + \zeta e^{2} + \zeta^{2} t + \zeta^{2} x g + \zeta^{2} y u,\quad 
		t \ge x + v,\quad 
		u \ge d h + s d + y r + w y + u z^{2},\\
		v \ge \zeta s + \zeta^{2}(g^{2} + t e + r t),\quad 
		w \ge \zeta s + \zeta^{2}(e g + x e + y t),\\
		x \ge \zeta d + \zeta^{2}(g + u),\quad 
		y \ge \zeta c + \zeta^{2}(g + t),\quad 
		z \ge \zeta c + \zeta^{2}(e + x),
	\end{gather*}
	it follows that the values of the generating functions at $\zeta=1/4.5238$ are at most the corresponding values $c,d,\dots,z$, following the same kind of argument as in Lemma \ref{lem:alternative}. As they are bounded, the growth rates of all the sequences are at most $4.5238$. Therefore,
	\[
		\lambda\le 4.5238.\qedhere
	\]
\end{proof}

\appendix
\section{Proof of Lemma \ref{lem:converge}}
\label{sec:converge}
	We prove by induction. The base case with $n=1$ is trivial. Assuming that it is true up to some $n$, we prove that it also holds for $n+1$. Let us go with the first one:
	\begin{align*}
		e_{n+1} &= x + x f_n \ge x\hat{E}(1) + x\sum_{i=1}^n \hat{F}(i) x^i
		= x\hat{E}(1) + \sum_{i=2}^{n+1} \hat{E}(i) x^i = \sum_{i=1}^{n+1} \hat{E}(i) x^{i},\\
		e_{n+1} &= x + x f_n \le x\hat{E}(1) + x\sum_{i=1}^\infty \hat{F}(i) x^i
		= x\hat{E}(1) + \sum_{i=2}^\infty \hat{E}(i) x^i = \sum_{i=1}^\infty \hat{E}(i) x^{i},
	\end{align*}
	where the induction hypothesis we use is
	\[
		\sum_{i=1}^n \hat{F}(i) x^i \le f_n \le \sum_{i=1}^\infty \hat{F}(i) x^i.
	\]

    We proceed with the sequence $f_n$:
    {\small
	\begin{align*}
		f_{n+1} &= x + x f_n + x g_n + x g_n l_n + g_n h_n \\
		&\ge x\hat{F}(1) + x\sum_{i=1}^n \hat{F}(i) x^i + x\sum_{i=1}^n \hat{G}(i) x^i + x\sum_{i=1}^n \hat{G}(i) x^i \sum_{j=1}^n \hat{L}(j) x^j  + \sum_{i=1}^n \hat{G}(i) x^i \sum_{j=1}^n \hat{H}(j) x^j \\
		&\ge x\hat{F}(1) + \sum_{i=2}^{n+1} \hat{F}(i-1) x^i + \sum_{i=2}^{n+1} \hat{G}(i-1) x^i + \sum_{k=2}^{n+1} \sum_{i+j=k-1} \hat{G}(i) \hat{L}(j) x^k + \sum_{k=2}^{n+1} \sum_{i+j=k} \hat{G}(i) \hat{H}(j) x^k \\
        &= x\hat{F}(1) + \sum_{k=2}^{n+1} \left(\hat{F}(k-1) +  \hat{G}(k-1) + \sum_{i+j=k-1} \hat{G}(i) \hat{L}(j) + \sum_{i+j=k} \hat{G}(i) \hat{H}(j)\right) x^k \\
        &= x\hat{F}(1) + \sum_{k=2}^{n+1} \hat{F}(k) x^k \\
		&= \sum_{k=1}^{n+1} \hat{F}(k) x^k.
	\end{align*}
    }
    We just remark that
    \[
        x\sum_{i=1}^n \hat{G}(i) x^i \sum_{j=1}^n \hat{L}(j) x^j\ge \sum_{k=3}^{n+1} \sum_{i+j=k-1} \hat{G}(i) \hat{L}(j) x^k = \sum_{k=2}^{n+1} \sum_{i+j=k-1} \hat{G}(i) \hat{L}(j) x^k
    \]
    since the term for $k=2$ is zero.

	The other inequality is carried out in almost the same way. We basically replace $n$ and $n+1$ by $\infty$ for the ranges:
    {\small
	\begin{align*}
		f_{n+1} &= x + x f_n + x g_n + x g_n l_n + g_n h_n \\
		&\le x\hat{F}(1) + x\sum_{i=1}^\infty \hat{F}(i) x^i + x\sum_{i=1}^\infty \hat{G}(i) x^i + x\sum_{i=1}^\infty \hat{G}(i) x^i \sum_{j=1}^\infty \hat{L}(j) x^j  + \sum_{i=1}^\infty \hat{G}(i) x^i \sum_{j=1}^\infty \hat{H}(j) x^j \\
		&\le x\hat{F}(1) + \sum_{i=2}^\infty \hat{F}(i-1) x^i + \sum_{i=2}^\infty \hat{G}(i-1) x^i + \sum_{k=2}^\infty \sum_{i+j=k-1} \hat{G}(i) \hat{L}(j) x^k + \sum_{k=2}^\infty \sum_{i+j=k} \hat{G}(i) \hat{H}(j) x^k \\
        &= x\hat{F}(1) + \sum_{k=2}^\infty \left(\hat{F}(k-1) +  \hat{G}(k-1) + \sum_{i+j=k-1} \hat{G}(i) \hat{L}(j) + \sum_{i+j=k} \hat{G}(i) \hat{H}(j)\right) x^k \\
        &= x\hat{F}(1) + \sum_{k=2}^\infty \hat{F}(k) x^k \\
		&= \sum_{k=1}^\infty \hat{F}(k) x^k.
	\end{align*}    
    }

	The treatment for the remaining sequences is carried out similarly with no new remarks; therefore, we omit the details.

\section{Verification of the inequalities in Lemma \ref{lem:full}} \label{sec:verification}
	The first one is trivial:
	\[
		C(n) = \myC_n = \begin{grid}
        \times & \square & \times \\
        \times & \square & \times \\
        \times & \times & \times
    \end{grid}_n =  \begin{grid}
        \times & \square & \times \\
        \times & \times & \times \\
        \times & \times & \times
		\end{grid}_{n-1} \le \begin{grid}
        \times & \square & \times \\
        \times & \times & \times \\
		\end{grid}_{n-1} = E(n-1). \\
	\]
	The following ones are similar by
	\[
		D(n) = \myD_n = \begin{grid}
        \times & \square & \void \\
        \times & \square & \times \\
	\times & \times & \times
    \end{grid}_n =  \begin{grid}
        \times & \square & \void \\
        \times & \times & \times \\
	\times & \times & \times
		\end{grid}_{n-1} \le G(n-1)
	\]
	and
	\[
		E(n) \le F(n-1)
	\]
	is already shown in Section \ref{sec:klarner-rivest}.

	The following three are actually equalities, but we keep them inequalities in the statement for consistency:
	\begin{align*}
		F(n) &= \myF_n = \myG_n + \myP_n = G(n) + P(n), \\
		G(n) &= \myG_n = \myE_n + \myQ_n = E(n) + Q(n), \\
		H(n) &= \myH_n = \myD_n + \myS_n = D(n) + S(n).
	\end{align*}

	We continue to verify the other two of similar nature:
	\begin{align*}
		T(n) &= \myT_n = \myX_n + \myV_n = X(n) + V(n), \\
		R(n) &= \myR_n = \myY_n + \myW_n = Y(n) + W(n).
	\end{align*}

	We then verify those with a bit harder nature with $3$ terms in the upper bound:
	\begin{align*}
		X(n) &= \myX_n \\
		&= \begin{grid}
        \void & \times & \void & \void \\
        \times & \square & \square & \times \\
        \times & \times & \times & \times
    \end{grid}_n + \begin{grid}
        \void & \square & \void & \void \\
        \times & \square & \square & \times \\
        \times & \times & \times & \times
		\end{grid}_n \\
		&\le \myD_{n-1} + \begin{grid}
        \void & \square & \times & \void \\
        \times & \square & \square & \times \\
        \times & \times & \times & \times
		\end{grid}_n + \begin{grid}
        \void & \square & \square & \void \\
        \times & \square & \square & \times \\
        \times & \times & \times & \times
		\end{grid}_n \\
		&\le D(n-1) + G(n-2) + U(n-2).
	\end{align*}
	We sketch the similar verification for $Y(n)$ and $Z(n)$:
	\begin{align*}
	Y(n) &= \myY_n \\
    &= \begin{grid}
        \times & \void & \times & \void \\
        \times & \square & \square & \times \\
        \times & \times & \times & \times
    \end{grid}_n + \begin{grid}
        \times & \times & \square & \void \\
        \times & \square & \square & \times \\
        \times & \times & \times & \times
    \end{grid}_n + \begin{grid}
        \times & \square & \square & \void \\
        \times & \square & \square & \times \\
        \times & \times & \times & \times
    \end{grid}_n \\
    &\le C(n-1) + G(n-2) + T(n-2), \\
	Z(n) &= \myZ_n \\
    &= \begin{grid}
        \times & \times & \void & \times \\
        \times & \square & \square & \times \\
        \times & \times & \times & \times
    \end{grid}_n + \begin{grid}
        \times & \square & \times & \times \\
        \times & \square & \square & \times \\
        \times & \times & \times & \times
    \end{grid}_n + \begin{grid}
        \times & \square & \square & \times \\
        \times & \square & \square & \times \\
        \times & \times & \times & \times
    \end{grid}_n \\
    &\le C(n-1) + E(n-2) + X(n-2).
	\end{align*}

	We continue with those having $4$ terms in the upper bound:
	\begin{align*}
		V(n) &= \myV_n \\
		    &=\begin{grid}
			\void & \times & \void & \void \\
			\times & \square & \square & \square \\
			\times & \times & \times & \times
		    \end{grid}_n + \begin{grid}
			\void & \square & \times & \void \\
			\times & \square & \square & \square \\
			\times & \times & \times & \times
		    \end{grid}_n + \begin{grid}
			\void & \square & \square & \times \\
			\times & \square & \square & \square \\
			\times & \times & \times & \times
		    \end{grid}_n + \begin{grid}
			\void & \square & \square & \square \\
			\times & \square & \square & \square \\
			\times & \times & \times & \times
		    \end{grid}_n \\
		    &\le S(n-1)
		      + \sum_{i+j=n-2} G(i)G(j)
		      + \sum_{i+j=n-2} T(i)E(j)
		      + \sum_{i+j=n-2} R(i)T(j).
	\end{align*}
	All the terms are obvious, except possibly the last one needing some more explanation. We first exclude two isolated squares (denoted as black squares in $\begin{grid}
			\void & \square & \square & \square \\
			\times & \blacksquare & \blacksquare & \square \\
			\times & \times & \times & \times
		    \end{grid}$), and decompose the polyomino into two smaller polyominoes with each containing 2 out of 4 remaining squares:
	\[
		 \begin{grid}
			\void & \square & \square & \square \\
			\times & \blacksquare & \blacksquare & \square \\
			\times & \times & \times & \times
		    \end{grid}_n = \begin{grid}
			\void & \mybox{u} & \mybox{v} & \mybox{s} \\
			\times & \times & \times & \mybox{r} \\
			\times & \times & \times & \times
		 \end{grid}_{n-2} \le \sum_{i+j=n-2} \begin{grid}
			 \void & \void & \void & \times & \void\\
			 \void & \mybox{u} & \mybox{v} & \times & \times \\
			 \times & \times & \times & \times & \times \\
			 \times & \times & \times & \times & \void
		 \end{grid}_i \begin{grid}
			\void & \times & \times & \void \\
			\times & \times & \times & \mybox{s} \\
			\times & \times & \times & \mybox{r} \\
			\times & \times & \times & \times
		 \end{grid}_j \le \sum_{i+j=n-2} R(i)T(j).
	\]
	The verification of $W(n)$ is carried out in a similar manner:
	\begin{align*}
		W(n) &= \myW_n \\
		&=\begin{grid}
        \times & \times & \void & \void \\
        \times & \square & \square & \square \\
        \times & \times & \times & \times
    \end{grid}_n + \begin{grid}
        \times & \square & \times & \void \\
        \times & \square & \square & \square \\
        \times & \times & \times & \times
    \end{grid}_n + \begin{grid}
        \times & \square & \square & \times \\
        \times & \square & \square & \square \\
        \times & \times & \times & \times
    \end{grid}_n + \begin{grid}
        \times & \square & \square & \square \\
        \times & \square & \square & \square \\
        \times & \times & \times & \times
    \end{grid}_n \\
		&\le S(n-1)
		      + \sum_{i+j=n-2} E(i)G(j)
		      + \sum_{i+j=n-2} X(i)E(j)
		      + \sum_{i+j=n-2} Y(i)T(j),
	\end{align*}
    where the last one is carried out just like the previous inequality:
   	\[
		 \begin{grid}
			\times & \square & \square & \square \\
			\times & \blacksquare & \blacksquare & \square \\
			\times & \times & \times & \times
		    \end{grid}_n = \begin{grid}
			\times & \mybox{u} & \mybox{v} & \mybox{s} \\
			\times & \times & \times & \mybox{r} \\
			\times & \times & \times & \times
		 \end{grid}_{n-2} \le \sum_{i+j=n-2} \begin{grid}
			 \void & \void & \void & \times & \void\\
			 \times & \mybox{u} & \mybox{v} & \times & \times \\
			 \times & \times & \times & \times & \times \\
			 \times & \times & \times & \times & \void
		 \end{grid}_i \begin{grid}
			\void & \times & \times & \void \\
			\times & \times & \times & \mybox{s} \\
			\times & \times & \times & \mybox{r} \\
			\times & \times & \times & \times
		 \end{grid}_j \le \sum_{i+j=n-2} Y(i)T(j).
	\] 
    
	Next ones are those with $5$ terms in the upper bound. We verify them in the increasing order of complexity:
	\begin{align*}
		Q(n) &= \myQ_n \\
        &= \begin{grid}
        \void & \times & \void \\
        \times & \square & \square \\
        \times & \times & \times
    \end{grid}_n + \begin{grid}
        \void & \square & \times \\
        \times & \square & \square \\
        \times & \times & \times
    \end{grid}_n + \begin{grid}
			\void & \square & \square & \void \\
			\times & \square & \square & \times \\
			\times & \times & \times & \void
    \end{grid}_n + \begin{grid}
			\void & \square & \square & \times \\
			\times & \square & \square & \square \\
			\times & \times & \times & \void
    \end{grid}_n + \begin{grid}
			\void & \square & \square & \square\\
			\times & \square & \square & \square\\
			\times & \times & \times & \void
    \end{grid}_n \\
		&\le G(n-1)
	      + \sum_{i+j=n-1} G(i)E(j)
	      + U(n-2)
	      + \sum_{i+j=n-2} T(i)G(j)
	      + \sum_{i+j=n-2} R(i)U(j),
    \end{align*}
    where the last term is derived likewise:
	\[
		 \begin{grid}
			\void & \square & \square & \square \\
			\times & \blacksquare & \blacksquare & \square \\
			\times & \times & \times & \void
		    \end{grid}_n = \begin{grid}
			\void & \mybox{u} & \mybox{v} & \mybox{s} \\
			\times & \times & \times & \mybox{r} \\
			\times & \times & \times & \void
		 \end{grid}_{n-2} \le \sum_{i+j=n-2} \begin{grid}
			 \void & \void & \void & \times & \void\\
			 \void & \mybox{u} & \mybox{v} & \times & \times \\
			 \times & \times & \times & \times & \times \\
			 \times & \times & \times & \times & \void
		 \end{grid}_i \begin{grid}
			\void & \times & \times & \void \\
			\times & \times & \times & \mybox{s} \\
			\times & \times & \times & \mybox{r} \\
			\times & \times & \times & \void
		 \end{grid}_j \le \sum_{i+j=n-2} R(i)U(j).
	\]

    \begin{align*}
	S(n) &= \myS_n \\
		&= \begin{grid}
        \times & \times & \void \\
        \times & \square & \square \\
        \times & \times & \times
    \end{grid}_n + \begin{grid}
        \times & \square & \times \\
        \times & \square & \square \\
        \times & \times & \times
    \end{grid}_n +  \begin{grid}
	    \times & \square & \square & \void \\
		\times & \square & \square & \times \\
			\times & \times & \times & \void
    \end{grid}_n +  \begin{grid}
	    \times & \square & \square & \times \\
		\times & \square & \square & \square \\
			\times & \times & \times & \void
    \end{grid}_n + \begin{grid}
	    \times & \square & \square & \square \\
		\times & \square & \square & \square \\
			\times & \times & \times & \void
    \end{grid}_n \\
	&\le G(n-1)
	      + \sum_{i+j=n-1} E(i)E(j)
	      + T(n-2)
	      + \sum_{i+j=n-2} X(i)G(j)
	      + \sum_{i+j=n-2} Y(i)U(j),
	\end{align*}
    where the last term is derived just as before:
   	\[
		 \begin{grid}
			\times & \square & \square & \square \\
			\times & \blacksquare & \blacksquare & \square \\
			\times & \times & \times & \void
		    \end{grid}_n = \begin{grid}
			\times & \mybox{u} & \mybox{v} & \mybox{s} \\
			\times & \times & \times & \mybox{r} \\
			\times & \times & \times & \void
		 \end{grid}_{n-2} \le \sum_{i+j=n-2} \begin{grid}
			 \void & \void & \void & \times & \void\\
			 \times & \mybox{u} & \mybox{v} & \times & \times \\
			 \times & \times & \times & \times & \times \\
			 \times & \times & \times & \times & \void
		 \end{grid}_i \begin{grid}
			\void & \times & \times & \void \\
			\times & \times & \times & \mybox{s} \\
			\times & \times & \times & \mybox{r} \\
			\times & \times & \times & \void
		 \end{grid}_j \le \sum_{i+j=n-2} Y(i)U(j).
	\]
    
		The last two sequences are the most complicated ones with $3$-fold convolutions:
	\begin{align*}
		P(n) &= \myP_n \\
			&= \begin{grid}
			\times & \void & \void \\
			\square & \square & \void \\
			\times & \times & \times
		    \end{grid}_n + \begin{grid}
			\square & \times & \void \\
			\square & \square & \void \\
			\times & \times & \times
		    \end{grid}_n + \begin{grid}
			\times & \void & \void \\
			\square & \square & \void \\
			\square & \square & \void \\
			\times & \times & \times
		    \end{grid}_n + \begin{grid}
			\square & \times & \void \\
			\square & \square & \void \\
			\square & \square & \void \\
			\times & \times & \times
		    \end{grid}_n + \begin{grid}
			\square & \square & \void \\
			\square & \square & \void \\
			\square & \square & \void \\
			\times & \times & \times
		    \end{grid}_n \\
		&\le \sum_{i+j=n} E(i)H(j)
	      + \sum_{i+j=n} Q(i)D(j)
	      + \sum_{i+j=n} X(i)R(j) \\
         &\quad + \sum_{i+j=n} V(i)Y(j)
	      + \sum_{i+j+k=n} U(i)Y(j)Z(k).
	\end{align*}
    The last term is the first time we use a $3$-fold convolution. We actually split the polyomino into $3$ polyominoes instead of $2$ polyominoes by
    \[
    \begin{grid}
			\mybox{u} & \mybox{v} & \void \\
			\mybox{s} & \mybox{x} & \void \\
			\mybox{r} & \mybox{a} & \void \\
			\times & \times & \times
    \end{grid}_n \le \sum_{i+j+k=n} \begin{grid}
			\void & \mybox{u} & \mybox{v} & \void \\
			\times & \times & \times & \times \\
			\times & \times & \times & \times \\
			\void & \times & \times & \times
    \end{grid}_i     \begin{grid}
			\void & \times & \times & \void \\
			\times & \times & \times & \times \\
			\void & \mybox{s} & \times & \times \\
			\void & \mybox{r} & \times & \times \\
			\void & \times & \times & \times
    \end{grid}_j     \begin{grid}
            \void & \times & \times & \void \\
			\times & \times & \times & \times \\
			\times & \times & \mybox{x} & \void \\
			\times & \times & \mybox{a} & \void \\
			\void & \times & \times & \times
    \end{grid}_k \le \sum_{i+j+k=n} U(i)Y(j)Z(k).
    \]
    The remaining is verified likewise:
\begin{align*}
			U(n) &= \myU_n \\
			&= \begin{grid}
			\void & \times & \void & \void \\
			\void & \square & \square & \void \\
			\times & \times & \times & \times
		    \end{grid}_n + \begin{grid}
			\void & \square & \times & \void \\
			\void & \square & \square & \void \\
			\times & \times & \times & \times
		    \end{grid}_n + \begin{grid}
			\void & \times & \void & \void \\
			\void & \square & \square & \void \\
			\void & \square & \square & \void \\
			\times & \times & \times & \times
		    \end{grid}_n + \begin{grid}
			\void & \square & \times & \void \\
			\void & \square & \square & \void \\
			\void & \square & \square & \void \\
			\times & \times & \times & \times
		    \end{grid}_n + \begin{grid}
			\void & \square & \square & \void \\
			\void & \square & \square & \void \\
			\void & \square & \square & \void \\
			\times & \times & \times & \times
		    \end{grid}_n \\
			&\le \sum_{i+j=n} D(i)H(j)
		      + \sum_{i+j=n} S(i)D(j)
		      + \sum_{i+j=n} Y(i)R(j) \\
		      & \quad + \sum_{i+j=n} W(i)Y(j)
		      + \sum_{i+j+k=n} U(i)Z(j)Z(k),
		\end{align*}
where the last term is due to
    \[
    \begin{grid}
			\void & \mybox{u} & \mybox{v} & \void \\
			\void & \mybox{s} & \mybox{x} & \void \\
			\void & \mybox{r} & \mybox{a} & \void \\
			\times & \times & \times & \times
    \end{grid}_n \le \sum_{i+j+k=n} \begin{grid}
			\void & \mybox{u} & \mybox{v} & \void \\
			\times & \times & \times & \times \\
			\times & \times & \times & \times \\
			\times & \times & \times & \times
    \end{grid}_i     \begin{grid}
			\void & \times & \times & \void \\
			\times & \times & \times & \times \\
			\void & \mybox{s} & \times & \times \\
			\void & \mybox{r} & \times & \times \\
			\times & \times & \times & \times
    \end{grid}_j     \begin{grid}
            \void & \times & \times & \void \\
			\times & \times & \times & \times \\
			\times & \times & \mybox{x} & \void \\
			\times & \times & \mybox{a} & \void \\
			\times & \times & \times & \times
    \end{grid}_k \le \sum_{i+j+k=n} U(i)Z(j)Z(k).
    \]

We have verified all the inequalities and hence proved Lemma \ref{lem:full}.

\bibliographystyle{unsrt}
\bibliography{fewtwigs}

\end{document}